\documentclass{amsart}
\usepackage{amsmath, amsthm, amscd, amsfonts, amssymb}
\usepackage{tikz}
\usetikzlibrary{matrix,arrows}

\makeatletter \oddsidemargin.9375in \evensidemargin \oddsidemargin
\newtheorem{thm}{Theorem}[section]
\newtheorem{lem}[thm]{Lemma}
\newtheorem{pro}[thm]{Proposition}

\theoremstyle{definition}
\newtheorem{den}[thm]{Definition}

\theoremstyle{remark}
\newtheorem{rem}[thm]{Remark}
\numberwithin{equation}{section}

\begin{document}

\title[Connes-amenability of enveloping dual Banach algebra] {On approximate Connes-amenability of enveloping\\ dual Banach algebras}
\author[A. Shirinkalam and A. Pourabbas ]{Ahmad Shirinkalam and  A. Pourabbas }
\address{Faculty of Mathematics and Computer
Science, Amirkabir University of Technology, 424 Hafez Avenue,
Tehran 15914, Iran}
\email{shirinkalam\_a@aut.ac.ir}

 \email{arpabbas@aut.ac.ir}



\begin{abstract}
For a Banach algebra $ \mathcal{A} $, we introduce various  approximate virtual diagonals such as approximate WAP-virtual diagonal and approximate virtual diagonal.

 For the enveloping dual Banach algebra $ F(\mathcal{A}) $ of $ \mathcal{A} $, we show that $ F(\mathcal{A}) $ is approximately Connes-amenable if and only if $ \mathcal{A} $ has an approximate WAP-virtual diagonal.
 
 Further,   for a discrete group $ G $, we show that if the group algebra $ \ell^1(G) $ has an approximate WAP-virtual diagonal,
 then it has 
  an approximate virtual diagonal.
\end{abstract}

\maketitle

\section{Introduction and Preliminaries}
The notion of approximate Connes-amenability for a dual Banach algebra was first introduced by  Esslamzadeh {\it et al.}  (2012). A dual Banach algebra 
$ \mathcal{A} $ is called approximately Connes-amenable if for every normal dual $ \mathcal{A} $-bimodule $ X $, every $ w^*-$continuous derivation $ D:\mathcal{A}\rightarrow X $ is approximately inner, that is, there exists a net $ (x_\alpha) \subseteq X $ such that $ D(b)= \lim_\alpha b\cdot x_\alpha -x_\alpha \cdot b \;(b\in \mathcal{A} ) $.

 Let $ \mathcal{A} $ be a Banach algebra and let $ X $ be a Banach $ \mathcal{A} $-bimodule. An element
 $ x \in X $ is called weakly almost periodic if the module maps $ \mathcal{A}\rightarrow X;\, a\mapsto a\cdot x $ and $ a \mapsto x\cdot a $ are weakly compact. The set of all weakly almost
 periodic elements of $X$ is  denoted by $ \mathrm{WAP}(X) $ which is a norm-closed
 subspace of $ X $. 

 For a Banach algebra $ \mathcal{A} $, one of the most important subspaces of $ \mathcal{A}^* $ is $ \mathrm{WAP}(\mathcal{A}^*) $ which is   left introverted  in the sense of \cite[\S 1]{LL}. 
  Runde (2004) observed that, the space $F( \mathcal{A})= \mathrm{WAP}(\mathcal{A}^*)^* $ is a dual Banach algebra with the first Arens product inherited from $ \mathcal{A}^{**} $. 
    He also showed that $ F(\mathcal{A}) $ is a \textit{canonical} dual Banach algebra associated to $ \mathcal{A} $ (see \cite{LL} or \cite[Theorem  4.10]{R5} for more details).
  
 Choi {\it et al.} (2014) called  $ F(\mathcal{A}) $
  the enveloping dual Banach algebra associated to $ \mathcal{A} $ and they studied Connes-amenability of $ F(\mathcal{A}) $. Indeed,
they introduced the notion of   WAP-virtual diagonal for a Banach algebra $ \mathcal{A} $ and they
showed that for a given Banach algebra $ \mathcal{A} $, the dual Banach algebra $ F(\mathcal{A}) $ is Connes-amenable if and only if $ \mathcal{A} $ admits a WAP-virtual diagonal.
They  also showed that for  a group algebra $ L^1(G) $, the existence of a virtual diagonal is equivalent to the existence of a WAP-virtual diagonal. As a  consequent $ L^1(G) $ is amenable if and only if $ F( L^1(G) ) $ is Connes-amenable, a fact that previously was shown by  Runde (2004).

Motivated by  these results, 
 we investigate an
\textit{approximate} analogue of  WAP-virtual diagonal related to approximate Conns-amenability of $ F(\mathcal{A}) $. 
Indeed, we introduce a notion of an approximate  WAP-virtual diagonal for $ \mathcal{A} $ and we  show that $ F(\mathcal{A}) $ is  approximately Connes-amenable if and only if $ \mathcal{A} $ has an approximate  WAP-virtual diagonal.
We also introduce various notions of  approximate-type  virtual diagonals  for a Banach algebra  -such as approximate virtual diagonal- and we  show that  
for a discrete group $ G $, if $ \ell^1(G) $ has an approximate WAP-virtual diagonal,
then it has an approximate virtual diagonal.

  Given a Banach algebra $ \mathcal{A} $, its unitization  is denoted by $ \mathcal{A}^\sharp $.   For a Banach $ \mathcal{A} $-bimodule $ X $, the topological dual space $ X^* $ of $  X$
becomes a Banach $\mathcal{A} $-bimodule via the following actions 
\begin{eqnarray}\label{action}
	 \langle x, a\cdot \varphi \rangle =\langle x\cdot a, \varphi\rangle, \qquad \langle x, \varphi\cdot a \rangle =\langle a\cdot x, \varphi\rangle \qquad  (a\in \mathcal{A}, x\in X, \varphi \in X^*).
\end{eqnarray}
 \begin{den} 
 	A Banach algebra $ \mathcal{A} $ is called dual, if it is a dual Banach space with a predual $ \mathcal{A}_* $ such that the  multiplication in $ \mathcal{A} $ is separately $ \sigma (\mathcal{A}, \mathcal{A}_*) $-continuous \cite [Definition 1.1]{Ru1}. Equivalently, a Banach algebra $ \mathcal{A} $ is dual if it has a (not necessarily unique) predual  $ \mathcal{A}_* $ which is a closed submodule of $ \mathcal{A}^* $ \cite [Exercise 4.4.1]{R}.
 \end{den}
 Let $ \mathcal{A} $ be a dual Banach algebra and  let a dual Banach space $ X $ be an $  \mathcal{A} $-bimodule. An element $ x\in X $ is called normal, if the module actions $ a\mapsto a\cdot x $ and $ a\mapsto x\cdot a $ are $ w^*\!\!-w^*$-continuous.
 The set of all normal elements in $ X $ is denoted by $ X_\sigma $. We say that $ X $ is normal if  $X=X_\sigma $. 
 %It is easy to see that $ X_\sigma $ is a norm-closed submodule of $ X $, however, there is no need  for $ X_\sigma $ to be $ w^*$-closed.

Recall that  for a Banach algebra $ \mathcal{A} $ and a Banach $ \mathcal{A}$-bimodule $ X $, a bounded linear map $  D:\mathcal{A}\rightarrow X $ is called a bounded derivation if $ D(ab)=a\cdot D(b)+D(a)\cdot b $ for every $ a,b \in \mathcal{A} $. A derivation $  D:\mathcal{A}\rightarrow X $ is called inner if there exists  an element $ x \in X $ such that for every $ a\in \mathcal{A} $, $D(a)=\mathrm{ad}_x(a)= a \cdot x - x \cdot a $.
 A dual Banach algebra $ \mathcal{A} $ is called Connes-amenable if for every normal dual $ \mathcal{A}$-bimodule $ X $, every $ w^*\!\!-w^*$-continuous derivation $ D:\mathcal{A}\rightarrow X $ is inner \cite [Definition 1.5]{R5}.

Following \cite{CSS}, for a given Banach algebra $ \mathcal{A} $ and a Banach $ \mathcal{A} $-bimodule $ X $, the $ \mathcal{A} $-bimodule $ \mathrm{WAP}(X^*) $ is denoted  by $ F_{\mathcal{A}}(X)_* $ and its dual space $ \mathrm{WAP}(X^*)^* $ is  denoted by $ F_{\mathcal{A}}(X)$.
 Choi {\it et al.}  \cite[Theorem  4.3]{CSS}  showed that   if $ X $ is a Banach $ \mathcal{A} $-bimodule, then   $ F_{\mathcal{A}}(X) $  is  a normal dual  $ F(\mathcal{A}) $-bimodule. Since $ \mathrm{WAP}(X^*) $ is a closed subspace of $ X^* $, we have a quotient map $q: X^{**} \twoheadrightarrow F_{\mathcal{A}}(X)$. Composing the canonical inclusion map  $ X \hookrightarrow X^{**} $ with $ q $, we obtain a continuous 
$ \mathcal{A} $-bimodule map $ \eta_X:X \rightarrow F_{\mathcal{A}}(X) $ which has a $w^*$-dense range. In the special case where $ X= \mathcal{A}$, we usually omit the subscript and simply use the notation $ F(\mathcal{A}) $ and the map $ \eta_\mathcal{A}:\mathcal{A} \rightarrow F({\mathcal{A}})$ is an algebra homomorphism.

It is well-known that for a given Banach  algebra $ \mathcal{A} $, the projective tensor product $ \mathcal{A}\hat{\otimes} \mathcal{A} $ is a 
 Banach $ \mathcal{A} $-bimodule through
$$ a \cdot (b \otimes c):=ab \otimes c,  \quad (b \otimes c)\cdot a := b \otimes ca, \quad (a, b, c\in \mathcal{A}) $$
 and  the map  $ \Delta:\mathcal{A}\hat{\otimes} \mathcal{A} \rightarrow \mathcal{A} $ defined on elementary tensors by $\Delta (a \otimes b)=ab  $ and extended by linearity and continuity, is an $ \mathcal{A} $-bimodule map with respect to the  module structure of $ \mathcal{A}\hat{\otimes} \mathcal{A} $. Let $ \Delta_{\mathrm{WAP}}:F_{\mathcal{A}}(\mathcal{A}\hat{\otimes} \mathcal{A})\rightarrow F(\mathcal{A}) $ be the $ w^*\!\!-w^*$-continuous  $ \mathcal{A} $-bimodule map induced by  $ \Delta:\mathcal{A}\hat{\otimes} \mathcal{A} \rightarrow \mathcal{A} $ (see \cite[Corollary  5.2]{CSS} for more details). 
\begin{den}\cite [Definition 6.4]{CSS}
An element $ M \in F_{ \mathcal{A} }(\mathcal{A}\hat{\otimes} \mathcal{A}) $ is called a $\mathrm{WAP}$-virtual diagonal for $  \mathcal{A} $ if 
$$a\cdot M= M\cdot a\quad  \text{and} \quad \Delta_{\mathrm{WAP}}(M)\cdot a=\eta_{\mathcal{A}} (a) \qquad (a\in \mathcal{A}).$$
\end{den}
Recall that an $ \mathcal{A} $-bimodule $ X $ is called neo-unital if every $ x \in X $ can be written as $ a \cdot y \cdot b $ for some $a, b \in  \mathcal{A}  $ and $ y \in X $.
\begin{rem}\label{ess}
	Let $ \mathcal{A} $ be a Banach algebra with a bounded approximate identity and let $ X $ be a Banach $ \mathcal{A} $-bimodule. Then by \cite[Proposition 1.8]{J}, the subspace $ X_{ess} := \mathrm{lin} \{a\cdot x\cdot b : a, b \in  \mathcal{A},\, x \in X   \}
	$ is a  neo-unital closed sub-$ \mathcal{A} $-bimodule of $ X $. Moreover, $ X_{ess}^{\perp} $ is complemented in $ X^* $. \end{rem}
\section{Approximate Connes-amenability of $ F(\mathcal{A}) $}
In this section we find  some conditions under  which $ F(\mathcal{A}) $ is  approximately Connes-amenable. 	For an arbitrary dual Banach algebra $ \mathcal{A} $, approximate Connes-amenability of $ \mathcal{A} $ is equivalent to approximate Connes-amenability of $ \mathcal{A}^\sharp $ \cite [Proposition 2.3]{ESA}, so without loss of generality throughout of this section  we may suppose that $ F(\mathcal{A}) $ has an identity element $ e $.
\begin{rem}\label{unit} Let $\mathcal{A} $ be a Banach algebra with an identity element $ e $.
Let $X $ be a normal dual $\mathcal{A} $-bimodule and let $ D:\mathcal{A} \rightarrow X $ be a 
bounded derivation.
By \cite[Lemma 2.3]{GL}, there exists a bounded derivation $ D_1 :\mathcal{A} \rightarrow e \cdot X\cdot e $ such that  $ D=D_1+\mathrm{ad}_x $ for some $ x \in X $. Furthermore, $D$ is inner derivation if and only if $D_1$ is inner.
\end{rem}
\begin{lem}\label{12}
Let $ \mathcal{A} $ be a Banach algebra. Then  $ F(\mathcal{A}) $ is approximately Connes-amenable if and only if every bounded derivation from $ \mathcal{A} $ into a unit-linked, normal dual $F(\mathcal{A}) $-bimodule is approximately inner.
\end{lem}
\begin{proof}
Suppose that $ F(\mathcal{A}) $ is approximately Connes-amenable. Let $ N $ be a unit-linked, normal dual $F(\mathcal{A}) $-bimodule and let $ D :\mathcal{A} \rightarrow N  $ be a bounded derivation. We show that $ D $ is approximately inner. By \cite[Theorem 4.4]{ESA}, there is a $ w^*\!\!-w^*$-continuous derivation $  \tilde{D} : F(\mathcal{A}) \rightarrow F_\mathcal{A} (N) $ such that $\tilde{D} \eta_\mathcal{A}= \eta_ND $, that is, the following diagram commutes
\begin{center}
\begin{tikzpicture}
	\matrix [matrix of math nodes,row sep=1cm,column sep=1cm,minimum width=1cm]
	{
		|(A)| \displaystyle \mathcal{A}  &   |(B)|   N     \\
		|(C)|     F(\mathcal{A})     &   |(D)|  F_\mathcal{A} (N). \\
	};
	\draw[->]    (A)-- node [above] { $ D $}(B);
	\draw[->]   (A)--  node [left] { $ \eta_\mathcal{A} $} (C);
	\draw[->]  (C)-- node [below]   { $ \tilde{D} $}(D);
	\draw[->]  (B)-- node [right]  {$ \eta_N $}(D);
	\end{tikzpicture}
\end{center}
 
 Since $ F_\mathcal{A} (N)  $ is a normal dual $F(\mathcal{A}) $-bimodule and  $ F(\mathcal{A}) $ is approximately Connes-amenable,  the derivation $ \tilde{D} $ 
is approximately inner. Since $ \eta_\mathcal{A} $ is an algebra homomorphism,   the derivation $ \tilde{D} \eta_\mathcal{A} $ is also approximately inner.
Hence the derivation $ \eta_ND $ is approximately inner, that is, there is a net $ (\xi_i)\subseteq F_\mathcal{A} (N) $
such that $$ \eta_ND (a)= \lim_i a\cdot\xi_i - \xi_i \cdot a \quad (a \in \mathcal{A}) .$$
By \cite[Corollary 5.3]{ESA} there is a 
$ w^*\!\!-w^*$-continuous $ \mathcal{A} $-bimodule map $ \epsilon_N :F_\mathcal{A} (N)\rightarrow N  $ such that $ \epsilon_N\eta_N= \mathrm{id}_N  $. Now by setting $ x_i= \epsilon_N(\xi_i) \in N$ for each $ i $  we have
$$D(a)=\epsilon_N\eta_ND(a)=\epsilon_N(\lim_i a\cdot\xi_i - \xi_i \cdot a)=\lim_i a\cdot x_i - x_i \cdot a,  $$ so $ D:\mathcal{A} \rightarrow N  $ is approximately inner.

Conversely, suppose that every bounded derivation from $ \mathcal{A} $ into a unit-linked, normal dual $F(\mathcal{A}) $-bimodule is approximately inner. Let $ N $ be a normal dual $F(\mathcal{A}) $-bimodule and let $ D:F(\mathcal{A}) \rightarrow N $ be a 
$ w^*\!\!-w^*$-continuous derivation.  By Remark \ref{unit} without loss of generality we may assume that $ N $ is a unit-linked, normal dual $F(\mathcal{A}) $-bimodule. We shall show that $ D $ is approximately inner. Since $ \eta_\mathcal{A} $ is an algebra homomorphism,
 the map $ D^\prime :=D\eta_\mathcal{A}:\mathcal{A} \rightarrow N $ is a bounded derivation, so
 by assumption $ D^\prime  $ is approximately inner.
 Since $ N $ is normal, $ D^\prime $ extends to a $ w^*\!\!-w^*$-continuous approximately inner derivation $  D^{\prime\prime} :F(\mathcal{A})\rightarrow N $. The derivations $ D $ and $  D^{\prime\prime} $ are $ w^*\!\!-w^*$-continuous and they  agree  on the $w^*-$dense subset $ \eta_\mathcal{A}(\mathcal{A})\subseteq  F(\mathcal{A})$, so they  agree on $ F(\mathcal{A}) $.
 Hence $ D $ is approximately inner, that is, $ F(\mathcal{A}) $ is approximately Connes-amenable.
\end{proof}
\begin{rem}\label{s}
Since $ F(\mathcal{A}) $ has an identity element $ e $,
by \cite [Theorem 6.8]{CSS} there exists an element $ s \in F_{\mathcal{A}}(\mathcal{A}\hat{\otimes} \mathcal{A}) $
such that $\Delta_{\mathrm{WAP}}(s)=e.$ We use this fact in the following lemma several times. 
\end{rem}
\begin{lem}\label{lemmm}
	Let $ \mathcal{A} $ be a Banach algebra. Then the following are equivalent
	\begin{enumerate}
		\item[(i)] $ F(\mathcal{A}) $ is approximately Connes-amenable;
		\item[(ii)] There exists a  net $ (M_\alpha)_\alpha \subseteq F_{\mathcal{A}}(\mathcal{A}\hat{\otimes} \mathcal{A})  $ such that for every $ a \in \mathcal{A} $, $ 	a \cdot M_\alpha - M_\alpha \cdot a \rightarrow 0 $ and $\Delta_{\mathrm{WAP}}(M_\alpha) = e  $ for all $ \alpha $.
	\end{enumerate}
\end{lem}
\begin{proof}
	(i)$ \Rightarrow$(ii)  Fix $ s\in F_{\mathcal{A}}(\mathcal{A}\hat{\otimes} \mathcal{A}) $ as mentioned in Remark \ref{s}.
		  We note that $ \ker\Delta_{\mathrm{WAP}}  $ is a normal dual $ F(\mathcal{A}) $-bimodule, now we define a $ w^*\!\!-w^*$-continuous derivation $ D: F(\mathcal{A}) \rightarrow \ker\Delta_{\mathrm{WAP}} $ by $ D(\textbf{a})=s \cdot \textbf{a} - \textbf{a} \cdot s \quad (\textbf{a} \in F(\mathcal{A}))$. Since $ F(\mathcal{A}) $ is approximately Connes-amenable, the derivation $D$  is approximately inner. Thus there is a net $ (N_\alpha) \subseteq \ker\Delta_{\mathrm{WAP}}  $ such that for every $ a \in \mathcal{A} $
	\begin{eqnarray}\label{lim}
		s \cdot a - a \cdot s=\lim_\alpha (a \cdot N_\alpha -N_\alpha \cdot a).
	\end{eqnarray}
	If we set  $ M_\alpha = N_\alpha +s $ for every
	$ \alpha $, then  we have $\Delta_{\mathrm{WAP}}(M_\alpha)  =   \Delta_{\mathrm{WAP}}(s)  
	=e $ and for every $a \in \mathcal{A}$ using (\ref{lim}) we have
	$$ a \cdot M_\alpha - M_\alpha \cdot a \rightarrow 0,$$
	 as required.
	
	(ii)$  \Rightarrow$(i) The hypothesis in (ii) ensures that each $ M_\alpha -s $ is in $ \ker\Delta_{\mathrm{WAP}} $ and for every $ a \in \mathcal{A} $
	\begin{eqnarray}\label{1}
		a\cdot (M_\alpha -s)-(M_\alpha -s) \cdot a \rightarrow s \cdot a - a \cdot s.
	\end{eqnarray}
	Let $ N $ be a unit-linked normal dual $ F(\mathcal{A}) $-bimodule and let $ D: \mathcal{A} \rightarrow N $ be a bounded derivation. We show that $ D $ is approximately inner, so by Lemma \ref{12}, $ F(\mathcal{A}) $ is approximately Connes-amenable.
Using the terminology of \cite [Theorem 6.8]{CSS}, if we define 
$ d(a)=s \cdot a - a \cdot s$ for all $a \in \mathcal{A}$, then $ d:\mathcal{A} \rightarrow \ker\Delta_{\mathrm{WAP}} $ is weakly universal for derivation $D$ with coefficient in $N$, that is,   there exists a $ w^*\!\!-w^*$-continuous  (and so it is norm continuous) $ \mathcal{A} $-bimodule map
	$ f :\ker\Delta_{\mathrm{WAP}} \rightarrow N $ such that $ fd=D $. Set $ y_\alpha=f(M_\alpha -s) $ for every $\alpha$.  Using (\ref{1}) for every $ a \in \mathcal{A} $ we have
	\begin{align*}
		D(a)
		&=fd(a)
		= f(s \cdot a - a  \cdot s )\\
		&=f(\lim _\alpha a\cdot (M_\alpha -s)-(M_\alpha -s) \cdot a ) \\
		&= \lim _\alpha f( a\cdot (M_\alpha -s)-(M_\alpha -s) \cdot a )  \\
		&=\lim _\alpha a \cdot y_\alpha -y_\alpha \cdot a.
	\end{align*}
	Hence $ D $ is approximately inner and this completes the proof.
\end{proof}
 Now we introduce a notion of an approximate WAP-virtual diagonal for $ \mathcal{A} $.
\begin{den}
	Let $ \mathcal{A} $ be a Banach algebra and let $ F(\mathcal{A}) $ has an identity element $ e $.
	An \textit{approximate $\mathrm{WAP}$-virtual diagonal}  for $ \mathcal{A} $ is a net $ (M_\alpha)_\alpha \subseteq F_{\mathcal{A}}(\mathcal{A}\hat{\otimes} \mathcal{A})  $ such that
	\begin{align}\label{deff}
	a \cdot M_\alpha - M_\alpha \cdot a \rightarrow 0\quad  (a \in \mathcal{A})  \quad \quad \hbox{and} \quad  \Delta_{\mathrm{WAP}}(M_\alpha) \rightarrow e
	\end{align}
 in the norm topology.
\end{den}
\begin{thm}\label{main} Let $ \mathcal{A} $ be a Banach algebra. Then the following are equivalent
	\begin{enumerate}
		\item[(i)] $ F(\mathcal{A}) $ is approximately Connes-amenable;
		
		\item[(ii)] $ \mathcal{A} $ has an approximate $\mathrm{WAP}$-virtual diagonal.
	\end{enumerate}
\end{thm}
\begin{proof}
(i)$  \Rightarrow$(ii) It is immediate by  Lemma \ref{lemmm}.
	
(ii)$  \Rightarrow$(i) Suppose that $ \mathcal{A} $ has an approximate $\mathrm{WAP}$-virtual diagonal. Then there exists a  net $ (M_\alpha)_\alpha \subseteq F_{\mathcal{A}}(\mathcal{A}\hat{\otimes} \mathcal{A})  $ such that for every $ a \in \mathcal{A} $, 
\begin{equation}\label{e-1}
a \cdot M_\alpha - M_\alpha \cdot a \rightarrow 0 \quad{\hbox {and}}\quad \Delta_{\mathrm{WAP}}(M_\alpha) \rightarrow e.
\end{equation}

Let $ N $ be a unit-linked normal dual $ F(\mathcal{A}) $-bimodule and let $ D:\mathcal{A} \rightarrow N  $ be a bounded derivation. We show that $ D $ is approximately inner and then by Lemma \ref{12}, $ F(\mathcal{A}) $ is approximately Connes-amenable.

 For every $ x \in N_* $ consider the functional $ f_x \in  F_{\mathcal{A}}(\mathcal{A}\hat{\otimes} \mathcal{A} )^* $ defined by $ f_x(a \otimes b)=(a\cdot D(b))(x) \quad (a, b \in \mathcal{A}) $.
   Note that for every $ m \in \mathcal{A}\hat{\otimes} \mathcal{A} $ we have
	$$ f_{x\cdot a-a\cdot x}(m)=(f_x\cdot a - a \cdot f_x)(m)+(\Delta(m)\cdot D(a))(x). $$ 
	Consider the $ \mathcal{A} $-bimodule map  $ \Delta:\mathcal{A}\hat{\otimes} \mathcal{A} \rightarrow \mathcal{A}.$ Then
by \cite [Corollary 5.2]{CSS} there exists a unique $w^*\!\!-w^*$-continuous linear map
$\Delta_{\mathrm{WAP}}: F_{\mathcal{A}}(\mathcal{A}\hat{\otimes} \mathcal{A})\to F (\mathcal{A})$
  makes the following diagram  
	\begin{center}
		\begin{tikzpicture}
		\matrix [matrix of math nodes,row sep=1cm,column sep=1cm,minimum width=1cm]
		{
			|(A)| \displaystyle \mathcal{A}\hat{\otimes} \mathcal{A}  &   |(B)|  \mathcal{A}     \\
			|(C)|     F_{\mathcal{A}}(\mathcal{A}\hat{\otimes} \mathcal{A})     &   |(D)|  F (\mathcal{A}) \\
		};
		\draw[->]    (A)-- node [above] { $ \Delta $}(B);
		\draw[->]   (A)--  node [left] { $ \eta_{\mathcal{A}\hat{\otimes} \mathcal{A}} $} (C);
		\draw[->]  (C)-- node [below]   { $ \Delta_{\mathrm{WAP}} $}(D);
		\draw[->]  (B)-- node [right]  {$ \eta_\mathcal{A}$}(D);
		\end{tikzpicture}
	\end{center}
commute.
Since $\eta_{\mathcal{A}\hat{\otimes} \mathcal{A}}  (\mathcal{A}\hat{\otimes} \mathcal{A}) $ is
$ w^*$-dense  in  $ F_{\mathcal{A}}(\mathcal{A}\hat{\otimes} \mathcal{A}) $,  there is a net $ (m^\alpha_\beta)_\beta \subseteq \mathcal{A}\hat{\otimes} \mathcal{A} $ such that 
$ \eta_{\mathcal{A}\hat{\otimes} \mathcal{A}}(m^\alpha_\beta)\xrightarrow{w^*}  M_\alpha  $ for every $ \alpha $. 

 Set $\lambda_\alpha(x)=\langle M_\alpha, f_x\rangle$   for each $ \alpha $, we have
	\begin{align*}
\langle a\cdot \lambda_\alpha - \lambda_\alpha \cdot a, x \rangle 
&=\langle \lambda_\alpha, x\cdot a -a\cdot x  \rangle\\
&= \langle M_\alpha, f_{x\cdot a-a\cdot x}\rangle\\
&=\lim_\beta \langle \eta_{\mathcal{A}\hat{\otimes} \mathcal{A}}(m^\alpha_\beta), f_{x\cdot a-a\cdot x}  \rangle\\
&= \lim_\beta \langle f_x\cdot a-a\cdot f_x,   m^\alpha_\beta \rangle + \lim_\beta \langle \Delta ( m^\alpha_\beta)\cdot D(a) ,x  \rangle \\
&= \lim_\beta \langle f_x\cdot a-a\cdot f_x,   m^\alpha_\beta \rangle + \lim_\beta \langle x , \eta_{\mathcal{A}} \Delta ( m^\alpha_\beta)\cdot D(a) \rangle \\
&= \lim_\beta \langle  \eta_{\mathcal{A}\hat{\otimes} \mathcal{A}}(m^\alpha_\beta), f_x\cdot a-a\cdot f_x     \rangle + \lim_\beta \langle x ,  \Delta_{\mathrm{WAP}} \eta_{\mathcal{A}\hat{\otimes} \mathcal{A}}( m^\alpha_\beta)\cdot D(a) \rangle \\
&=  \langle a\cdot M_\alpha - M_\alpha \cdot a, f_{x}   \rangle +   \langle  x , \Delta_{\mathrm{WAP}}  (M_\alpha)\cdot D(a) \rangle.
\end{align*}	
Because $\| f_x \| \leq \| D \| \| x \|   $,
\begin{equation*}
\Arrowvert (a\cdot \lambda_\alpha - \lambda_\alpha \cdot a)( x) -D(a)(x)\Arrowvert\leq \Arrowvert a \cdot M_\alpha - M_\alpha \cdot a\Arrowvert  \Arrowvert D \Arrowvert  \Arrowvert x \Arrowvert + \Arrowvert x \Arrowvert  \Arrowvert \Delta_{\mathrm{WAP}}  (M_\alpha) - e \Arrowvert  \Arrowvert D(a)\Arrowvert.
\end{equation*}
This together with equation (\ref{e-1})	imply that  $ D(a)=\lim_\alpha \mathrm{ad}_{\lambda_\alpha}(a)$ for every $a \in \mathcal{A}$, so $ D $ is approximately inner.
\end{proof}
% % % % % % % % % % % % % % % % % % % % % % % % % % % % % % % % % % % % % % % % % %
\section{Approximate-type  virtual diagonals }
Suppose that  $ \mathcal{A} $ is a Banach algebra. Consider the $ \mathcal{A} $-bimodule  map $ \Delta^*:\mathcal{A}^*\rightarrow (\mathcal{A}\hat{\otimes} \mathcal{A})^*  $. Let $ V\subseteq  (\mathcal{A}\hat{\otimes} \mathcal{A})^*$ be a closed subspace  and let $ E\subseteq(\Delta^*)^{-1}(V)\subseteq \mathcal{A}^* $. Then we obtain a map $  \Delta^*\vert_{E}:E \rightarrow V$. We denote the adjoint of $  \Delta^*\vert_{E} $ by $ \Delta_E: V^* \rightarrow E^*. $
\begin{den}
 Let $ V\subseteq  (\mathcal{A}\hat{\otimes} \mathcal{A})^*$ be a non-zero closed subspace and let $ E\subseteq(\Delta^*)^{-1}(V)\subseteq \mathcal{A}^* $.  An \textit{approximate $ V $-virtual diagonal} for $ \mathcal{A} $ is a net $( M_\alpha)\subseteq V^* $ satisfies 
$$ a\cdot M_\alpha -M_\alpha\cdot a \rightarrow 0\qquad \text{and}\quad  \langle \Delta_E(M_\alpha)\cdot a,\varphi\rangle \rightarrow \langle\varphi,a\rangle \qquad (a\in \mathcal{A}, \:\varphi \in E).$$
\end{den}
We consider the following cases
	\begin{enumerate}
\item[1-] Let $ V=(\mathcal{A}\hat{\otimes} \mathcal{A})^* $. Then   we obtain a new concept of a diagonal for $\mathcal{A}  $, called an \textit{approximate virtual diagonal}, that is, a -not necessarily bounded-  net $( M_\alpha)\subseteq(\mathcal{A}\hat{\otimes} \mathcal{A})^{**} $ such that
$$ a\cdot M_\alpha -M_\alpha\cdot a \rightarrow 0\qquad \text{and}\quad  \langle \Delta^{**}(M_\alpha)\cdot a,\varphi\rangle \rightarrow \langle\varphi,a\rangle \qquad (a\in \mathcal{A}, \varphi \in \mathcal{A}^*).$$
\item[2-]Let $ V= (\mathcal{A}\hat{\otimes} \mathcal{A})^*_{ess} \subseteq  (\mathcal{A}\hat{\otimes} \mathcal{A})^*  $. Then we obtain an \textit{approximate $ (\mathcal{A}\hat{\otimes} \mathcal{A})^*_{ess} $-virtual diagonal} for $A$.
	\end{enumerate}
In the following proposition we show that 
how these two diagonals deal with.
\begin{pro}\label{pro}
	Let $ \mathcal{A} $ be a Banach algebra with a bounded approximate identity. If $\mathcal{A}  $ has an  approximate $ (\mathcal{A}\hat{\otimes} \mathcal{A})^*_{ess} $-virtual diagonal, then it has an approximate virtual diagonal.
\end{pro}
\begin{proof}
	Let $ V=(\mathcal{A}\hat{\otimes} \mathcal{A})^*_{ess} $.	By Remark \ref{ess}, $V  $  is a closed subspace of $(\mathcal{A}\hat{\otimes} \mathcal{A})^*  $ and we have an isomorphism $ \tau : V^* \rightarrow (\mathcal{A}\hat{\otimes} \mathcal{A})^{**}/V^\bot$ defined by $ \tau (T)=x+V^\bot $ for some $ x \in (\mathcal{A}\hat{\otimes} \mathcal{A})^{**} $.
	 Since $V^\bot$  is complemented in $(\mathcal{A}\hat{\otimes} \mathcal{A})^{**  }$, 
 	there exists an $ \mathcal{A} $-bimodule projection $ P:(\mathcal{A}\hat{\otimes} \mathcal{A})^{**  } \rightarrow V^\bot $ and so there is  an $ \mathcal{A} $-bimodule isomorphism   $ \iota :(I-P)(\mathcal{A}\hat{\otimes} \mathcal{A})^{**  }\rightarrow V^* $. 	
 	For every $v \in V$ and $T \in V^*$ by definition of $\tau$ we have   
 	 $\tau (T)=x+V^\bot $ for some $ x \in (\mathcal{A}\hat{\otimes} \mathcal{A})^{**} $, thus
 	  $$\langle \iota^{-1}(T),v\rangle =\langle (I-P)x,v  \rangle=\langle x,v  \rangle - \langle  P x, v\rangle=\langle x,v  \rangle=\langle x+V^\bot,v  \rangle=\langle \tau (T),v  \rangle,   $$ that is, $ \iota^{-1}(T)|_{V} =T$ for every $ T\in V^* $.
 	 	 Now suppose that $ (F_\alpha)\subseteq V^* $ is an approximate $ (\mathcal{A}\hat{\otimes} \mathcal{A})^*_{ess} $-virtual diagonal for $ \mathcal{A} $ and
	set $ M_\alpha=\iota^{-1}(F_\alpha)\in (\mathcal{A}\hat{\otimes} \mathcal{A})^{**  } $ for every $\alpha$.
	Then
	\begin{eqnarray}\label{00}
	a \cdot M_\alpha - M_\alpha \cdot a = \iota^{-1}(a \cdot  F_\alpha - F_\alpha \cdot a) \rightarrow 0.
	\end{eqnarray}
	 Moreover, for $ a \in \mathcal{A} $ by Cohen's factorization theorem, there exist 
	 $ x, b, y \in \mathcal{A}$ such that
	$ a= ybx $. Since  for every $\psi \in \mathcal{A}^*$ we have  $bx\cdot \psi  \cdot y \in \mathcal{A}^*_{ess}  $ and since 
	 $\mathcal{A}^*_{ess}=(\Delta^*)^{-1}(V)   $ \cite [Lemma 7.7]{CSS},  $ \Delta^*(bx\cdot \psi  \cdot y)\in V $. But $ M_\alpha |_{V}=\iota^{-1}(F_\alpha)|_{V}=F_\alpha $  for each $ \alpha $ . Using (\ref{00})
	\begin{equation}\label{e-2}
	\begin{split}
		\langle \Delta ^{**}(M_\alpha)\cdot a , \psi \rangle - \langle \Delta^{**}(F_\alpha) \cdot b,  x\cdot \psi  \cdot y\rangle
	&=\langle \Delta ^{**}(M_\alpha)\cdot a , \psi \rangle - \langle F_\alpha, \Delta^* (bx\cdot \psi  \cdot y)\rangle\\
	&= \langle \Delta ^{**}(M_\alpha)\cdot a , \psi \rangle - \langle M_\alpha, \Delta^* (bx\cdot \psi  \cdot y)\rangle\\
	&=\langle \Delta ^{**}(M_\alpha)\cdot a , \psi \rangle - \langle \Delta ^{**}(M_\alpha),  bx\cdot \psi  \cdot y\rangle\\
	&= \langle \Delta ^{**}(M_\alpha)\cdot a , \psi \rangle - \langle y \cdot \Delta ^{**}( M_\alpha)\cdot bx,  \psi  \rangle \rightarrow 0.
	\end{split}
	\end{equation}	
	Also since $ (F_\alpha) $ is an approximate $ (\mathcal{A}\hat{\otimes} \mathcal{A})^*_{ess} $-virtual diagonal and since $ \Delta^*(x \cdot \psi \cdot y)\in V $, we have
\begin{equation}\label{e-3}	 \langle \Delta ^{**}(F_\alpha)\cdot b , x \cdot\psi \cdot y  \rangle \rightarrow \langle x \cdot \psi \cdot y , b\rangle. \end{equation}
	Hence (\ref{e-2}) and (\ref{e-3}) imply that
	\begin{equation}\label{e-4}
	\begin{split}
			| \langle \Delta ^{**}(M_\alpha)\cdot a , \psi \rangle - \langle \psi , a\rangle |
	&\leq |\langle \Delta ^{**}(M_\alpha)\cdot a , \psi \rangle - \langle \Delta^{**}(F_\alpha) \cdot b,  x\cdot \psi  \cdot y\rangle|\\
	&+ | \langle \Delta^{**}(F_\alpha) \cdot b,  x\cdot \psi  \cdot y\rangle - \langle x \cdot \psi \cdot y , b\rangle|\to 0.
\end{split}	\end{equation}
	We conclude from (\ref{00}) and (\ref{e-4}) that
	$$	a \cdot M_\alpha - M_\alpha \cdot a \rightarrow 0,\qquad	 \langle \Delta ^{**}(M_\alpha)\cdot a , \psi \rangle \to \langle \psi , a\rangle, $$
which means that $ (M_\alpha) $ is an  approximate virtual diagonal for $ \mathcal{A} $.
\end{proof}

As a corollary of Lemma \ref{12}, if a Banach algebra $ \mathcal{A} $ is approximately amenable, then $ F(\mathcal{A}) $ is approximately Connes-amenable.
For the rest of the paper we consider the converse of this  result in the special case, whenever $ \mathcal{A}=\ell^1(G) $. Indeed,
 we show that for a discrete group $ G $, if $ \ell^1(G) $ has an approximate WAP-virtual diagonal,
then it has an approximate virtual diagonal, although the existence of  an approximate virtual diagonal is weaker than approximate amenability of $ \ell^1(G) $.

 We identify $ L^1(G )\hat{\otimes}L^1(G )  $ with $ L^1 (G\times G) $  as $ L^1(G) $-bimodules (see \cite [Example VI. 14]{BD} for instance).

Let $ I $  be  the closed ideal in $   L^\infty (G\times G)_{ess} $ generated  by $\Delta ^*(C_0(G))  $ as in \cite [Definition 8.1]{CSS}, that is,
$$ I=\overline{\hbox{lin}\{ \Delta ^*(f)g : f \in C_0(G), g \in   L^\infty (G\times G)_{ess} \}}. $$
\begin{rem}\label{re}
 The space of uniformly continuous bounded functions on $ G $ is denoted by $ \mathrm{UCB}(G) $. 
Since $ \Delta^*(\mathrm{UCB}(G)) = \Delta^* (L^\infty (G)_{ess})  \hookrightarrow L^\infty (G\times G)_{ess} $, we obtain  a map $ \Delta_{\mathrm{UCB}}:  L^\infty (G\times G)_{ess} ^* \rightarrow \mathrm{UCB}(G)^* $ as the adjoint of the inclusion map.
By \cite [Lemma 8.2]{CSS}, $ \Delta ^*(C_0(G))\subseteq I \subseteq  \mathrm{WAP}( L^\infty (G\times G))  $, so we obtain a map $ \Delta_{C_0}=(\Delta^*|_{C_0(G)})^*:I^* \rightarrow C_0(G)^*=M(G) $.
 By \cite [Proposition  8.5]{CSS} there is a $ G_d $-bimodule map $ S:I^* \rightarrow   L^\infty (G\times G)_{ess} ^* $ such that for every $ h \in \mathrm{UCB}(G) $ and $ \psi \in I^* $
$$\langle  \Delta_{\mathrm{UCB}} (S(\psi)),h \rangle=\langle   \iota(\Delta_{C_0}(\psi)), h\rangle,  $$
where $ \iota :M(G)\rightarrow \mathrm{UCB}(G)^* $ is the natural inclusion map defined by $ \langle \iota(\mu),h\rangle=\int_G h\:d\mu$ for every $\mu \in M(G)$ and  $h \in \mathrm{UCB}(G)$.
\end{rem}
\begin{thm}
Let $ G $	be a discrete group. Consider the following
statements
\begin{enumerate}
		\item[(i)]$ \ell^1(G) $ has an  approximate $\mathrm{WAP} $-virtual diagonal;
		\item[(ii)] $ \ell^1(G) $ has an  approximate $  \ell^\infty (G\times G)_{ess}  $-virtual diagonal;
		\item[(iii)] $ \ell^1(G) $ has an  approximate virtual diagonal.
	\end{enumerate}
\end{thm}	
Then the implications (i)$ \Rightarrow $(ii)$ \Rightarrow $(iii) hold. 
\begin{proof}	
(i)$ \Rightarrow $(ii)
 Let $ (M_\alpha)\subseteq F(\ell^1 (G\times G)) $ be an approximate WAP-virtual diagonal for 
$ \ell^1(G) $.
 Then for every $ a \in \ell^1(G) $
$$a \cdot M_\alpha - M_\alpha \cdot a \rightarrow 0  \quad (a \in \ell^1(G)) \quad   \hbox{and}   \quad   \Delta _{\mathrm{WAP}}(M_\alpha)\rightarrow  e, $$ where $ e $ denotes the identity element of  $ F(\ell^1(G)) $.
Since  $ \Delta ^*(c_0(G))\subseteq I \subseteq  \mathrm{WAP}( \ell^\infty (G\times G)) $, each $ M_\alpha $ can be restricted to a functional on $ I $. Let $ S:I^* \rightarrow   \ell^\infty (G\times G)_{ess} ^*  $ be as mentioned  in Remark \ref{re}.
 Set $ N_\alpha =S(M_\alpha |_{I}) $.
 We show that $ (N_\alpha)_\alpha $ is an approximate  $  \ell^\infty (G\times G)_{ess}  $-virtual diagonal.
Since $ G $ is discrete and  $ S $ is a continuous $G  $-bimodule map, it is a continuous $\ell^1(G)  $-bimodule map, so
$$ a \cdot N_\alpha - N_\alpha \cdot a = S(a \cdot M_\alpha - M_\alpha \cdot a)\rightarrow 0 \quad ( a \in \ell^1(G) ). $$
Since  $ (M_\alpha) $ is an  approximate WAP-virtual diagonal $ \Delta_{c_0}(M_\alpha |_{I})\cdot a\rightarrow a $  in $ \ell^1(G) $ and since the map $ \iota :\ell^1(G)\rightarrow \mathrm{UCB}(G)^* $ is $ w^*\!\!-w^*$-continuous, for every $ h \in \mathrm{UCB}(G)  $ we have
\begin{align*}
\langle h , a \rangle
	&= \lim_\alpha \langle  \iota \Delta_{c_0}(M_\alpha |_{I}\cdot a) , h \rangle\\
	&= \lim_\alpha \langle  \Delta_{\mathrm{UCB}}S(M_\alpha |_{I}\cdot a) , h \rangle\\
	&= \lim_\alpha \langle  \Delta_{\mathrm{UCB}}(N_\alpha \cdot a) , h \rangle.
\end{align*}
Hence $ (N_\alpha)_\alpha $ is an approximate  $  \ell^\infty (G\times G)_{ess}  $-virtual diagonal for $ \ell^1(G)  $.

(ii) $\Rightarrow $(iii)  holds by Proposition \ref{pro}.
\end{proof}
Note that,  it is not clear for us that 	 the map $ S $ in Remark \ref{re} is always an $ L^1(G) $-bimodule map. If $ S $ is  an $ L^1(G) $-bimodule map, then the  previous theorem holds  not only for discrete groups but also for each 
 locally compact group $ G $.

%%%%%%%%%%%%%%%%%%%%%% % % % % % % % % % % % % % % % % %% % END OF NEW RESULT


\begin{thebibliography}{99}
\bibitem{BD}
F. F. Bonsall, J. Duncan, {\it Complete Normed Algebra}, Vol. 80, Springer-Verlag, New York, 1973.
\bibitem{CSS}
 Y. Choi, E. Samei and R. Stokke, {\it Extension of derivations, and Connes-amenability of the enveloping dual Banach algebra},  arxiv:math.FA/1307.6287v2, 2014.
\bibitem{ESA}
G. H. Esslamzadeh, B. Shojaee and A. Mahmoodi,  {\it Approximate Connes-amenability of dual Banach algebras}, Bull. Belg. Math. Soc. Simon stevin {\bf 19} (2012), 193-213.
\bibitem{GL}
F. Ghahramani and R. J. Loy ,  {\it Generalized notions of amenability},  J. Funct. Anal. {\bf 208} (2004), 229-260.
\bibitem{J}
B. E. Johnson, {\it Cohomology in Banach algebras}, Mem.  Amer. Math. Soc. {\bf127} (1972).
\bibitem{LL}
A.T.-M Lau and R. J. Loy, {\it Weak amenability of Banach algebras on locally compact groups}, J. Funct. Anal. {\bf 145} (1997), 175-204.
\bibitem{Ru1}
V. Runde, {\it Amenability for dual Banach algebras}, Studia Math. {\bf 148} (2001), 47-66.
\bibitem{R5}
--------, {\it Dual Banach algebras: Connes-amenability, normal, virtual diagonals, and injectivity of the predual bimodule},
 Math. Scand. {\bf 95} (2004), 124-144.
\bibitem{R}
--------, {\it Lectures on Amenability}, Lecture Notes in Mathematics, Vol. 1774, Springer-Verlag, Berlin, 2002.
\end{thebibliography}
\end{document}